%% file: lsd_of_truncated_unutary.tex
\documentclass{amsart}

\usepackage{amssymb}
\usepackage{amsmath}
\usepackage{amsthm}
\usepackage{amsbsy}
\usepackage{bm}
\usepackage{hyperref}
\date{\today}
\usepackage{cite}
\usepackage{array}

\input{mypreamble}

\title[LSD of the product of truncated Haar unitary matrices]{Limiting spectral distribution of the product of truncated Haar unitary matrices}

\author{Kartick Adhikari}
\address{Stat-Math Unit\\
        Indian Statistical Institute\\
        Kolkata 700108, India}

\email{kartickmath [at] gmail.com}

\author{Arup Bose}
\address{Stat-Math Unit\\
        Indian Statistical Institute\\
        Kolkata 700108, India}
        
\email{bosearu [at] gmail.com}

\date{\today}
\thanks{The work of Kartick Adhikari is partially supported by National Post-Doctoral Fellowship, India, with reference no. {\bf PDF/2016/001601}.  The work of Arup Bose is  supported by J. C. Bose National Fellowship, Department of Science and Technology, Government of India}
\begin{document}
\begin{abstract}
Let $A_m^{(1)},\ldots, A_m^{(k)}$ be $m\times m$ left-uppermost blocks of $k$ independent $n\times n$ Haar unitary matrices 
where $\frac{n}{m}\to \alpha$ as $m\to \infty$,  with $1<\alpha<\infty$.
Using free probability and  Brown measure techniques,  we find the limiting spectral distribution of $A_m^{(1)}\cdots A_m^{(k)}$.  
\end{abstract}

\maketitle

\section{Introduction and main results}

Let $\mathcal U_n$ be the compact group of $n\times n$ unitary matrices. The Haar probability
measure $\nu_n$ on $\mathcal U_n$  is bi-invariant. An $n\times n$  unitary random matrix $U_n$ is said
to be {\it a Haar unitary matrix} if its distribution is $\nu_n$.

 Let $A_n$ be an $n\times n$ random matrix with eigenvalues $\lambda_1,\ldots, \lambda_n$. Then 
its {\it empirical spectral measure} of $A_n$  is the probability measure 
$\frac{1}{n}\sum_{k=1}^n\d_{\lambda_k}$, where $\d_x$ is the Dirac delta measure at $x$. Equivalently, its {\it empirical spectral distribution} (ESD) is given by
\begin{align*}
F^{A_n}(x,y)=\frac{|\{k\;:\; \Re(\lambda_k)\le x, \Im(\lambda_k)\le y\}|}{n},\;\;\;\mbox{ for $x,y\in \R$},
\end{align*} 
where $|\cdot|$ denotes the cardinality and $\Im(x)$ and $\Re (x)$ denote, respectively, the imaginary and real parts of $x$.
Clearly $F^{A_n}$ is a random distribution function. 
If, as $n \to \infty$,  it converges (almost surely) to a non-random distribution function $F_{\infty}$ weakly, then this limit is said to be the almost sure {\it limiting spectral distribution} (LSD) of $A_n$. The {\it expected} ESD function $\E[F^{A_n}(x,y)]$ is a non-random  distribution function. Its limit is also called the LSD and coincides with the earlier limit if both exist.

Let $A_m$ be the $m\times m$ left-uppermost sub-matrix of an $n\times n$  Haar distributed unitary matrix $U_n$. It is known that (see \cite[Theorem 5]{jiang}) that the ESD of the properly scaled eigenvalues of $A_m$, for $m=o(\sqrt n),$ converges to the circular law in probability.

In this article we look at a different regime of $m$-values. 
Suppose that $\frac{n}{m}\to \alpha
$ as $m\to \infty$, where $1<\alpha<\infty$. Our main result provides the LSD of $A_m$  and more generally the LSD of the product of finitely many such independent matrices. 
\begin{theorem}\label{thm:1}
Let $n_1,\ldots , n_{k+1}$ be $k+1$ positive integers such that $n_1=n_{k+1}=\min\{n_1,\ldots, n_{k+1}\}$. Suppose that 
$\frac{n}{n_i}\to \alpha_i$, where $1<\alpha_i<\infty$, for $i=1,\ldots, k$ as $n\to \infty$. Let $U_n^{(1)},\ldots, U_n^{(k)}$ be $k$ independent $n\times n$ Haar unitary matrices, and $A_1,\ldots, A_k$ be the $k$ left-uppermost sub-matrices of these matrices of dimensions $n_1\times n_2, \ldots, n_k\times n_{k+1}$, respectively. Then the LSD of $A_1\cdots A_k$ is $\mu_k$ almost surely, where $\mu_k$ is rotationally invariant and the distribution of its radial part is given by
\begin{align*}
\mu_k(\{z\;:\; |z|\le t\})=1+S^{<-1>}(t^{-2}),\;\mbox{for $t\le \frac{1}{\sqrt{\alpha_1\cdots \alpha_k}}$},
\end{align*}
where $ S(z)=\prod\limits_{i=1}^k\frac{\alpha_i(\alpha_1+z)}{\alpha_1+\alpha_iz}$ and  $f^{<-1>}$ denotes the  inverse of $f$ under the composition mapping. 

In particular, if $\alpha_1=\cdots= \alpha_k=\alpha$  then the LSD  of $A_1\cdots A_k$ is $\mu_k$ almost surely, and is given by
$$
d\mu_k(z)=\frac{(\alpha-1)}{k\pi}\frac{r^{\frac{2}{k}-1}}{(1-r^{\frac{2}{k}})^2}drd\theta,
$$
where $z=re^{i\theta}$ for $0\le r\le (\frac{1}{\alpha})^{\frac{k}{2}}$ and $0\le \theta\le 2\pi$. 
\end{theorem}

Since the regimes  in \cite[Theorem 5]{jiang} and  Theorem \ref{thm:1} are completely different, the methods of  
proofs are also so.  
We have the following corollary. 

\begin{corollary}\label{cor:1}
Let $m$ and $n$ be two positive integers such that $\frac{n}{m}\to \alpha$ as $n\to \infty$, where $1<\alpha<\infty$. Let $A_m^{(1)},\ldots, A_m^{(k)}$ be $m\times m$ left-uppermost  sub-matrices of $k$ independent $n\times n$ Haar unitary matrices $U_n^{(1)},\ldots, U_n^{(k)}$ respectively. Then the limiting distribution of the square of the radial part of the eigenvalues of $A_m^{(1)}\cdots A_m^{(k)}$ is almost surely same as the distribution of
 $$
 \l(\frac{U}{\alpha-1+U}\r)^k,
 $$
 where $U$ is uniformly distributed random variable on interval $[0,1]$. 
In particular, when $\alpha=1$, the  LSD of $A_m^{(1)}\cdots A_m^{(k)}$ is the uniform distribution on the unit circle. 
\end{corollary}
The LSD for the expected ESD of the squares of the radial part of the eigenvalues of $A_m^{(1)}\cdots A_m^{(k)}$ has been established in \cite[Theorem 22]{aihp}. They use the joint distribution of the eigenvalues of $A_m^{(1)}\cdots A_m^{(k)}$ for the proof. Instead we use free probability and Brown measure techniques. 

\section{preliminaries}
We first recall some basic definitions from the literature of free probability and Brown measure. A {\it non-commutative probability (NCP) space} is a pair $(\mathcal A,\vp)$ where $\mathcal A$ is a unital algebra over complex numbers and $\vp$ is a linear functional on $\mathcal A$ such that $\vp(\one_{\mathcal A})=1$. In addition, suppose that $\mathcal A$ is a $*$-algebra, i.e. that $\mathcal A$ is also endowed with an antilinear $*$-operation, $*:a\to a^*\in \mathcal A$, such that $(a^*)^*=a$ and $(ab)^*=b^*a^*$ for all $a,b\in \mathcal A$. Also if $\vp(aa^*)\ge 0$ for all $a\in \mathcal A$ then $(\mathcal A, \vp)$ is said to be a {\it $*$-probability space}.

Let $\mathcal A_n$ be the algebra of  $n\times n $ random matrices whose entries have all moments finite. It is equipped with the tracial state 
\begin{align*}
\vp_n(B_n)&:=\frac{1}{n}\Tr(B_n), \;\mbox{where $\Tr(B_n)=\sum_{i=1}^{n}b_{ii}$ when $B_n=(b_{ij})_{n\times n}\in \mathcal A_n$}.
\end{align*}
 We say that a sequence of random matrices $(A_n)$ from  $\mathcal A_n$ {\it converges in $*$-distribution almost surely } to some element $a\in \mathcal A$  
if for every choice of $\epsilon_1,\epsilon_2,\ldots, \epsilon_k\in \{1, *\}$ we have 
$$
\lim_{n\to \infty}\varphi_n(A_n^{\epsilon_1}\cdots A_n^{\epsilon_k})=\varphi(a^{\epsilon_1}\cdots a^{\epsilon_k}), \;\; \mbox{almost surely}.
$$
Then we write $A_n\stackrel{*\mbox{-dist}}{\longrightarrow}a$ almost surely as $n\to \infty$.

 Let $(a_i^n)_{i\in I}$ be a collection of random variables from $\mathcal A_n$ which converges in $*$-distribution to some $(a_i)_{i\in I}$ in  $(\mathcal A, \vp)$. Then $(a_i^{(n)})_{i\in I}$ are said to be asymptotically free if $(a_i)_{i\in I}$ are free.

Haar unitary elements and $R$-diagonal elements play a crucial role in the proof of our results.
 An element $u\in \mathcal A$ is said to be {\it Haar unitary} if it is a unitary (i.e. if $uu^*=u^*u=1$) and if $\vp(u^k)=0$, for all $k\in \Z\backslash \{0\}.$ For $k=0$ we have $\vp(u^0)=\vp(\one_{\mathcal A})=1$. Observe that this gives complete information about the $*$-distribution of $u$ because any $*$-moments of $u$ can be reduced to a moment of the form $\vp(u^k)$ for $k\in \Z$.

Let $\kappa_n(a_1,\ldots, a_n)$ denote the order $n$ free cumulant of $(a_1,\ldots, a_n)\in \mathcal A$ (see \cite[p. 175]{speicherbook}). An element $a\in \mathcal A$ is called {\it $R$-diagonal} if $\kappa_n(a_1,\ldots, a_n)=0$ for all $n\in \N$ whenever the arguments $a_1,\ldots, a_n\in \{a,a^*\}$ are not alternating in $a$ and $a^*$. As an example, $\kappa_4(a,a^*,a,a^*)$ is alternating. It is a convention that the cumulants with an odd number of arguments are not alternating, e.g. $\kappa_5(a^*,a,a^*,a,a^*)$ is not alternating. It is known that Haar unitary elements are $R$-diagonal.  For more details on $R$-diagonal elements we refer to  \cite[Lecture 15]{speicherbook}.

Next we introduce the Brown measure. Let  $\Delta (a) $ denote the Fuglede-Kadison  determinant (see \cite{fuglede}) of $a\in \mathcal A$. Then
\begin{align*}
\Delta (a):= \exp[\frac{1}{2}\vp(\log (aa^*)) ],
\end{align*}	
if $a$ is invertible. If $a$ is not invertible, then $\Delta (a):=\lim_{\e\to 0}\Delta_{\e} (a)$, where 
$$
\Delta_{\e} (a)=\exp[\frac{1}{2}\vp(\log (aa^*+\e^2)) ], \mbox{ for $\e>0$}.
$$
The Brown measure of $a\in \mathcal A$ is defined by (see \cite{brown}), for $\lambda\in \C$,
\begin{eqnarray*}
\mu_a &=&\frac{1}{2\pi}\left(\frac{\partial^2}{\partial(\Re \lambda)^2}+\frac{\partial^2}{\partial(\Im \lambda)^2}\r)\log\Delta (a-\lambda)\\
&=&\frac{2}{\pi}\frac{\partial}{\partial\lambda}\frac{\partial}{\partial\bar \lambda}\log\Delta (a-\lambda).
\end{eqnarray*}
One can show that in fact $\mu_a$ is a probability measure on $\C$. Consider any $n\times n$ matrix $A_n$. Then 
\begin{align*}
\Delta(A_n)=\sqrt[n]{|\det A_n|}, \mbox{ and } \mu_{A_n}=\frac{1}{n}\sum_{i=1}^n\d_{\lambda_i}
\end{align*}
 where $\lambda_1,\ldots, \lambda_n$ are the eigenvalues of $A_n$. So the Brown measure is the ESD of the matrix. See \cite{sniady} for details. 

However, if the ESD converges, there is no guarantee that the limiting Brown measure is the Brown measure of the limit element.
But often they do equal each other. For example the i.i.d. matrix converges in $*$-distribution to the circular element and its LSD is the uniform distribution on the unit disc. The latter is indeed the Brown measure of the circular element. See \cite{haagerup}. 
The LSD of  the elliptic matrix is the uniform probability measure on  an ellipse. At the same time, the Brown measure of an elliptic element is also the uniform probability measure on an ellipse (see \cite{belinschi}, \cite{larsen}). As a final example, in  \cite{manjusinglering} it has been shown that the LSD of a bi-unitary invariant random matrix as $n \to \infty$ is actually the Brown measure of the limit element. 

The Brown measure of any $R$-diagonal element can be described in terms of its $S$-transform (see Fact \ref{ft:brownR}).  Let $a\in \mathcal A$ such that $\vp(a)\neq 0$. Its moment generating series is defined as $$M_a(z)=\sum_{n=1}^{\infty}\vp(a^n)z^n.$$ 
Its {\it $S$-transform} is defined by 
$$S_a(z):=\frac{1+z}{z}M_a^{<-1>}(z),$$ 
where $f^{<-1>}$ denotes the  inverse of $f$ under the composition mapping (see \cite{speicherbook}, page 294). The moment generating series  and the $S$-transform are analytic functions on suitably chosen domains in the complex plane. One can show that $S_a$ is well defined in some neighbourhood of the origin.

\section{Proofs}\label{sec:proofs}

We will use the following fact whose proof can be found in \cite{haagerup}.

\begin{fact}\label{ft:brownR}
	Suppose $x$ is an $R$-diagonal element. Then its Brown measure $\mu_x$ is rotationally invariant and can be described by the probabilities  
	\begin{align*}
	\mu_x(\{\lambda\in \C \;:\; |\lambda|\le t\})=\left\{\begin{array}{lcr}
	0 & \mbox{for} & t\le \frac{1}{\sqrt{\vp(xx^*)^{-1}}}\\
	1+S_{xx^*}^{<-1>}(t^{-2}) & \mbox{for} & \frac{1}{\sqrt{\vp(xx^*)^{-1}}}\le t \le \sqrt{\vp(xx^*)}\\
	1 & \mbox{for } & t\ge \sqrt{\vp(xx^*)},
	\end{array} \right.
	\end{align*}
	where $f^{<-1>}$ denotes the  inverse of $f$ under the composition mapping.
\end{fact}

We first show that the $*$-limit of $A_1\cdots A_k$  is an $R$-diagonal element (see Lemma \ref{lem:rdiagonal}). Then we identify its $S$-transform  and invoke the above fact to finish the proof.

\begin{lemma}\label{lem:rdiagonal}
Let $A_1,\ldots, A_k$ be as defined in Theorem \ref{thm:1}. Then the limiting element in the sense of $*$-distribution of $A_1\cdots A_k$ is $R$-diagonal.
\end{lemma}

\begin{proof}[Proof of Lemma \ref{lem:rdiagonal}]
Let $A_{p\times q}$ be a $p\times q$ left-uppermost sub-matrix of an $n\times n$ Haar unitary matrix $U_n$. Let $V_p$ and $W_q$ be two non-random unitary matrices of dimensions $p\times p$ and $q\times q$, respectively. 
Let \begin{align*}
 	\widehat{V}_p&:=\l[\begin{array}{cc}
V_p & 0\\ 0 & I_{n-p}
\end{array}\r]_{n\times n}.
\end{align*}
and define $\widehat{W}_q$ likewise. Note that these matrices are $n\times n$ unitary matrices. 
Moreover, we have 
$$
\widehat{V}_p U_n \widehat{W}_q = \l[\begin{array}{cc}
V_pA_{p\times q} W_q & V_p B \\ CW_q & D
\end{array} \r]_{n\times n}, \mbox{ where }
U_n=\l[\begin{array}{cc}
A_{p\times q} & B\\ C & D
\end{array} \r]_{n\times n}.
$$
Since $U_n$ is bi-unitary invariant, we know that $\widehat{V}_p U_n \widehat{W}_q$ has the same distribution as $U_n$. Therefore $V_pA_{p\times q} W_q$ has the same distribution as $A_{p\times q}$.

Let $V_{n_1}$ be an $n_1\times n_1$ non-random unitary matrix, and $W_{n_i}$ be $n_i\times n_i$ non-random unitary matrices for $i=1,\ldots, k$.  Since $A_1,\ldots, A_k$ are independent, 
 $$
 V_{n_1}A_1\cdots A_kW_{n_1}=V_{n_1}A_1W_{n_2}W_{n_2}^*A_{2}\cdots W_{n_k}^*A_kW_{n_1}\stackrel{d}{=}A_1\cdots A_k.
 $$
 Therefore $A_1\cdots A_k$ is bi-unitary invariant. Hence the result follows from the fact that the limit of bi-unitary invariant matrices is $R$-diagonal (see \cite[Theorem 4.4.5]{hiaibook}).
\end{proof}

In the next lemma we calculate the $S$-transform of the limiting element of $A_1\cdots A_k$.

\begin{lemma}\label{lem:stransform}
	Let $A_1,\ldots, A_k$ be as in Theorem \ref{thm:1}, and $A_1\cdots A_k\stackrel{*\mbox{-\tiny{dist}}}{\longrightarrow} {a} $ almost surely, for some $a\in (\mathcal A,\vp)$,  with respect to $\vp_{n_1}$ as $n\to \infty$. Then the $S$-transform of $aa^*$ is given by
	$$
	S_{aa^*}(\lambda)=\prod_{i=1}^k\frac{\alpha_i(\lambda+\alpha_1)}{\alpha_1+\alpha_i\lambda},\;\mbox{ for all $\lambda\in \D$},
	$$
where $\D$ denotes the unit disk in the complex plane. If $\alpha_1=\cdots=\alpha_k=\alpha$ then 
	$$
	S_{aa^*}(\lambda)=\l(\frac{\lambda+\alpha}{1+\lambda}\r)^k,\;\mbox{ for all $\lambda\in \D$}.
	$$
\end{lemma}

The following facts will be used in the proofs of Lemma \ref{lem:stransform} and Theorem \ref{thm:1}.  

\begin{fact}(\cite[Theorem 23.13]{speicherbook},\cite[Theorem 4.3.1]{hiaibook})\label{ft:free}
	Let, for each $n\in \N$, $U_n^{(1)},\ldots, U_n^{(p)}$ be $p$ independent $n\times n$ Haar unitary random matrices. Let $D_n^{(1)},\ldots, D_n^{(q)}$ be $q$  constant matrices which converge in $*$-distribution (with respect to $\vp_n$) for $n\to \infty$, i.e.,
	$$
	D_n^{(1)},\ldots, D_n^{(q)}\stackrel{*\mbox{\tiny -dist}}{\longrightarrow} d_1,\ldots, d_q
	$$
	for some $d_1,\ldots, d_q\in (\mathcal A,\vp)$. Then, almost surely,
	\begin{align*}
	U_n^{(1)},\ldots,U_n^{(p)}, D_n^{(1)},\ldots, D_n^{(q)}\stackrel{*\mbox{\tiny -dist}}{\longrightarrow} u_1,\ldots,u_p,d_1,\ldots,d_q,
	\end{align*}
	where $u_1,\ldots,u_p,\{d_1,\ldots,d_q\}$ are free and where each $u_i$ is a Haar unitary element.
\end{fact}



\begin{fact}\cite[Corollary 18.17]{speicherbook}\label{ft:stransform}
	Let $(\mathcal A,\vp)$ be a non-commutative probability space, and let $a,b$ be in $\mathcal A$ such that $\vp(a), \vp(b)\neq 0$. If $a$ and  $b$ are free, then:
	\begin{align*}
		S_{ab}(\lambda)=S_a(\lambda)S_b(\lambda),
	\end{align*}
	in the neighbourhood of the origin in the complex plane where $S_a(\lambda)$ and $S_b(\lambda)$ are well defined.
\end{fact}

\begin{fact}\cite[p. 78]{speicherbook}\label{ft:freewithu}
Let $(\mathcal A, \vp)$ be a $*$-probability space. Consider a unital sub-algebra $\mathcal B\subset \mathcal A$ and a Haar unitary $u\in \mathcal A$ such that $\{u,u^*\}$ and $\mathcal B$ are free. Let $u\mathcal B u^*=\{ubu^*: \; b\in \mathcal B\}$. Then $\mathcal B$ and $u\mathcal B u^*$ are free.
\end{fact}

We use the following notation: for a $m\times m$ matrix $B_m$,
\begin{align*}
\widetilde {B}_m&:=\l[\begin{array}{cc}
B_m & 0\\ 0 & 0
\end{array}\r]_{n\times n}.
\end{align*}

\begin{proof}[Proof of Lemma \ref{lem:stransform}]
Suppose $B_{n_1}=A_1\cdots A_k$. Let $\widetilde B_{n_1}\stackrel{*\mbox{\tiny -dist}}{\longrightarrow }\tilde a$ almost surely, for some $\tilde a \in (\widetilde{\mathcal A}, \tilde \vp)$, with respect to $\vp_n$ as $n\to \infty$. Then, for $\e_1,\ldots, \e_k\in \{1, *\}$, 
\begin{eqnarray*}
\vp(a^{\e_1}\cdots a^{\e_k})&=&\lim_{n\to \infty }\vp_{n_1}(B_{n_1}^{\e_1}\cdots B_{n_1}^{\e_p})\\
&=&\lim_{n\to \infty}\frac{n}{n_1}\vp_n(\widetilde B_{n_1}^{\e_1}\cdots \widetilde B_{n_1}^{\e_p})\\
&=&\alpha_1 \tilde\vp (\tilde a^{\e_1}\cdots \tilde a^{\e_k}).
\end{eqnarray*}
	In particular, we have 
	\begin{align*}
	\vp((aa^*)^k)=\alpha_1 \tilde\vp((\tilde a \tilde a^*)^k),\ \text{for all positive integers}\ k.
	\end{align*}
 Therefore, for $\lambda\in \C$, 
	$$
	M_{aa^*}(\lambda)=\sum_{k=1}^{\infty}\vp((aa^*)^k)\lambda^k=\alpha_1\sum_{k=1}^{\infty}\tilde\vp((\tilde a\tilde a^*)^k)\lambda^k=\alpha_1 M_{\tilde a\tilde a^*}(\lambda).
	$$
	Hence $M_{aa^*}^{<-1>}(\lambda)=M_{\tilde a\tilde a^*}^{<-1>}(\frac{\lambda}{\alpha_1})$. Therefore the $S$-transform of $aa^*$ is given by
	\begin{align}\label{eqn:1}
	S_{aa^*}(\lambda)=\frac{1+\lambda}{\lambda}M_{aa^*}^{<-1>}(\lambda)=\frac{1+\lambda}{\lambda}M_{\tilde a\tilde a^*}^{<-1>}\l(\frac{\lambda}{\alpha_1}\r)=\frac{1+\lambda}{\alpha_1+\lambda}S_{\tilde a\tilde a^*}\l(\frac{\lambda}{\alpha_1}\r).
	\end{align}
 Note that $\widetilde I_{n_1},\ldots, \widetilde I_{n_k}$ converge jointly with respect to $\vp_n$ as $n\to \infty$, i.e.,
	$$\widetilde I_{n_1},\ldots, \widetilde I_{n_k}\stackrel{*\mbox{\tiny -dist}}{\longrightarrow} b_1,\ldots, b_k
	$$ 
 where  $b_1,\ldots, b_k\in (\widetilde{\mathcal A},\tilde \vp)$ are such that $\tilde \vp(b_i^{p})=\frac{1}{\alpha_i}$ for all $p\in \N$ and $i=1,\ldots,k$.
 Therefore, by Fact \ref{ft:free}, we have
 \begin{align*}
 U_n^{(1)},\ldots, U_n^{(k)}, \widetilde{I}_{n_1},\ldots, \widetilde I_{n_k}\stackrel{*\mbox{\tiny -dist}}{\longrightarrow} u_1,\ldots,u_k,b_1,\ldots,b_k\;\mbox{ almost surely,}
 \end{align*}
 where $u_1,\ldots, u_k$ are Haar unitary and free with $\{b_1,\dots,b_k\}$. In particular, we get
	\begin{align*}
	\widetilde B_{n_1}=\widetilde I_{n_1}U_n^{(1)}\widetilde I_{n_2}U_n^{(2)}\widetilde I_{n_3}\cdots \widetilde I_{n_k}U_n^{(k)}\widetilde I_{n_1}\stackrel{*\mbox{\tiny -dist}}{\longrightarrow}
	b_1u_1b_2u_2\cdots b_ku_kb_1 \;\mbox{almost surely},
	\end{align*}
	 Therefore we obtain 
	 \begin{align*}
	 \tilde a=b_1u_1b_2u_2\cdots b_ku_kb_1 \mbox{ and } \tilde a\tilde a^*=b_1u_1b_2u_2\cdots b_ku_kb_1u_k^*b_k\cdots u_2^*b_2u_1^*b_1.
	  \end{align*}

	Now we calculate the $S$-transform. For the ease of writing and for clarity, we restrict to the case $k=2$. Applying  Facts \ref{ft:stransform} and \ref{ft:freewithu} repeatedly, we have
	\begin{eqnarray*}
S_{b_1u_1b_2u_2b_1u_2^*b_2u_1^*b_1}(\lambda)&	= &S_{b_1^2}(\lambda)S_{u_1b_2u_2b_1u_2^*b_2u_1^*}(\lambda)
	\\ &=&S_{b_1}(\lambda)S_{b_2u_2b_1u_2^*b_2}(\lambda)\\ &=&S_{b_1}(\lambda)S_{b_2^2}(\lambda)S_{u_2b_1u_2^*}(\lambda)
	\\ &=&S_{b_1}(\lambda)S_{b_2}(\lambda)S_{b_1}(\lambda)
	\end{eqnarray*}
	Similarly, applying  Facts \ref{ft:stransform} and \ref{ft:freewithu}  repeatedly, we get 
	\begin{align}\label{eqn:2}
		S_{\tilde a\tilde a^*}(\lambda)=S_{b_1}(\lambda)S_{b_2}(\lambda)\cdots S_{b_k}(\lambda)S_{b_1}(\lambda).
	\end{align}
	It is easy to see that, for $i=1,2,\ldots, k $ and $\lambda\in \D$, we have  
	\begin{align}\label{eqn:3}
	M_{b_i}(\lambda)=\frac{\lambda}{\alpha_i(1-\lambda)}, \;\;
	M_{b_i}^{<-1>}(\lambda)=\frac{\alpha_i \lambda}{1+\alpha_i \lambda}, \mbox{ and } S_{b_i}(\lambda)=\frac{\alpha_i(1+\lambda)}{1+\alpha_i \lambda}.
	\end{align}
	Therefore, using \eqref{eqn:2} and \eqref{eqn:3} in \eqref{eqn:1}, we get 
	$$
	S_{aa^*}(\lambda)=\prod_{i=1}^{k}\frac{\alpha_i(\lambda+\alpha_1 )}{\alpha_1+\alpha_i\lambda},\;\mbox{ for all $\lambda \in \D$}.
	$$ 
	Hence, the lemma is proved.
\end{proof}

Finally we proceed to prove Theorem \ref{thm:1}.  

\begin{proof}[Proof of Theorem \ref{thm:1}]
Let $a$ and $\tilde a$ be as in the proof of Lemma \ref{lem:stransform}.
Then
\begin{align*}
\vp(a^{\e_1}\cdots a^{\e_k})
=\alpha \tilde\vp (\tilde a^{\e_1}\cdots \tilde a^{\e_k}).
\end{align*}
  Therefore, applying Fact \ref{ft:freewithu} repeatedly, we have
\begin{eqnarray}\label{eqn:mean}
\vp(\tilde a \tilde a^*)&=&\vp(b_1u_1b_2u_2\cdots b_ku_kb_1u_k^*b_{k}\cdots u_2^*b_2u_1^*b_1)\nonumber \\
&=&\frac{1}{\alpha_1\cdots \alpha_k\alpha_1}.
\end{eqnarray}

Since the Brown measure and the LSD of bi-unitary invariant matrices are same (see \cite[Remark 8]{manjusinglering}), it is enough to calculate the Brown measure of $a$. However,  $a$ is $R$-diagonal by Lemma \ref{lem:rdiagonal}. Therefore Fact \ref{ft:brownR} implies that the Brown measure $\mu_k$ of $a$ is rotationally invariant, and the distribution of its radial part is given by
\begin{align*}
\mu_k(\{z\suchthat |z|\le t\})=1+S_{aa^*}^{<-1>}(t^{-2}),\; \mbox{ for $\frac{1}{\sqrt{\vp((aa^*)^{-1})}}\le t\le \sqrt{\vp(aa^*)}$}.
\end{align*} 
 Again we have $\vp((aa^*)^{-1})=\infty$ and $\vp(aa^*)=\alpha_{1}\tilde\vp(\tilde a \tilde a^*)=\frac{1}{\alpha_1\cdots \alpha_k}$ (from \eqref{eqn:mean}).   Hence the result folllows upon  using Lemma \ref{lem:stransform}.
  
Now suppose that $\alpha_1=\cdots=\alpha_k=\alpha$. Then we have
\begin{align*}
S_{aa^*}(\lambda)=\l(\frac{\alpha+\lambda}{\lambda+1}\r)^k,\;\mbox{and }\; S_{aa^*}^{<-1>}(\lambda)=\frac{\alpha -\lambda^{\frac{1}{k}}}{\lambda^{\frac{1}{k}}-1}.
\end{align*}
Hence the distribution of the radial part is 
\begin{eqnarray*}
\mu_a(\{z\in \C\;:\; |z|\le t\})&=&1+S_{aa^*}^{<-1>}(t^{-2})\\
&=&\frac{(\alpha -1)t^{\frac{2}{k}}}{1-t^{\frac{2}{k}}},\;\mbox{ for $0\le t\le (\frac{1}{\alpha})^{\frac{k}{2}}$}.
\end{eqnarray*} 
 Therefore  the density of the radial part of the LSD is
\begin{align*}
\frac{2(\alpha-1)}{k}\frac{t^{\frac{2}{k}-1}}{(1-t^{\frac{2}{k}})^2}, \;\mbox{ for $0\le t\le (\frac{1}{\alpha})^{\frac{k}{2}}$}.
\end{align*}
The result now follows from the fact that  $\mu_k$ is rotationally invariant.
\end{proof}

\begin{proof}[Proof of Corollary \ref{cor:1}]
 Let $R_k$ denote the distribution  of the radial part of the LSD of $A_m^{(1)}\cdots A_m^{(k)}$. Then, by Theorem \ref{thm:1}, 
\begin{align*}
\P(R_k\le t)=\l\{\begin{array}{lcr}
\frac{(\alpha -1)t^{\frac{2}{k}}}{1-t^{\frac{2}{k}}} & \mbox{for} &0\le t\le (\frac{1}{\alpha})^{\frac{k}{2}}\\
& \\
1 & \mbox{for} & t \ge (\frac{1}{\alpha})^{\frac{k}{2}}.
\end{array}\r.
\end{align*}
Let $U$ be the uniform random variable in $[0,1]$. Then, for $ 0\le t\le (\frac{1}{\alpha})^{\frac{k}{2}}$, we have
\begin{eqnarray*}
\P(R_k^2\le t)&=&\frac{(\alpha -1)t^{\frac{1}{k}}}{1-t^{\frac{1}{k}}}\\
&=&\P\big(U\le \frac{(\alpha -1)t^{\frac{2}{k}}}{1-t^{\frac{2}{k}}}\big)\\
&=&\P\big(\big(\frac{U}{\alpha-1+U}\big)^k\le t\big).
\end{eqnarray*}
Hence the result.
\end{proof}

\bibliographystyle{amsplain}

\end{document}

%% file: mypreamble.tex
\theoremstyle{theorem}
    \newtheorem{theorem}{Theorem}
    \newtheorem{lemma}{Lemma}
    
    \newtheorem{corollary}[theorem]{Corollary}

\theoremstyle{definition} 
    
    \newtheorem{fact}{Fact}
    
    \newtheorem{remark}{Remark}
    \newtheorem{example}[theorem]{Example}
    \newtheorem{exercise}[theorem]{Exercise}

\def\suchthat{\; : \;}

\def\d{\delta}
\def\e{\epsilon}

\def\vp{\varphi}

\def\C{\mathbb{C}}
\def\D{\mathbb{D}}

\def\Z{\mathbb{Z}}

\def\N{\mathbb{N}}

\def\R{\mathbb{R}}

\def\Z{\mathbb{Z}}

\def\l{\left}
\def\r{\right}
\def\<{\langle}
\def\>{\rangle}

\newcommand{\one}{{\mathbf 1}}
\newcommand{\E}{\mbox{\bf E}}

\def\bar{\overline}
\def\P{{\bf P}}

\newcommand\Tr{{\mbox{Tr}}}

\newcommand\mnote[1]{} 
\newcommand\be{\begin{equation*}}

\newcommand\ee{\end{equation*}}

\newcommand\ben{\begin{equation}}
\newcommand\een{\end{equation}}
\newcommand\bes{\begin{eqnarray*}}
\newcommand\ees{\end{eqnarray*}}

\newcommand\bex{\begin{exercise}}
\newcommand\eex{\end{exercise}}
\newcommand\beg{\begin{example}}
\newcommand\eeg{\end{example}}
\newcommand\benu{\begin{enumerate}}
\newcommand\eenu{\end{enumerate}}
\newcommand\beit{\begin{itemize}}
\newcommand\eeit{\end{itemize}}
\newcommand\berk{\begin{remark}}
\newcommand\eerk{\end{remark}}
\newcommand\bdefn{\begin{defintion}}
\newcommand\edefn{\end{definition}}
\newcommand\bthm{\begin{theorem}}
\newcommand\ethm{\end{theorem}}
\newcommand\bprf{\begin{proof}}
\newcommand\eprf{\end{proof}}
\newcommand\blem{\begin{lemma}}
\newcommand\elem{\end{lemma}}

\newcommand{\sm}{{\raise0.3ex\hbox{$\scriptstyle \setminus$}}}

\def\l{\left}
\def\r{\right}

\def\CHI{\mathchoice%
{\raise2pt\hbox{$\chi$}}%
{\raise2pt\hbox{$\chi$}}%
{\raise1.3pt\hbox{$\scriptstyle\chi$}}%
{\raise0.8pt\hbox{$\scriptscriptstyle\chi$}}}
\def\smalloplus{\raise1pt\hbox{$\,\scriptstyle \oplus\;$}}